\documentclass[11pt,twoside]{amsart}

\usepackage{amssymb,latexsym,amscd,comment}

\newtheorem{thm}{Theorem}[section]
\newtheorem{prop}[thm]{Proposition}
\newtheorem{lem}[thm]{Lemma}
\newtheorem{cor}[thm]{Corollary}

\newtheorem{question}[thm]{Question}

\theoremstyle{definition}

\newtheorem{example}[thm]{Example}
\newtheorem{rem}[thm]{Remark}

\def\cocoa{{\hbox{\rm C\kern-.13em o\kern-.07em C\kern-.13em o\kern-.15em A}}}

\theoremstyle{remark}

\newcommand{\skipit}[1]{{}}
\newcommand{\prfend}{\hbox to7pt{\hfil}
\par\vskip-\baselineskip\hbox to\hsize
{\hfil\vbox {\hrule width6pt height6pt}}\vskip\baselineskip}

\newcommand {\PP}{\mathbb{P}}

\newcommand {\FF}{\mathbb{F}}

\newcommand {\X} {\mathbb {X}}

\DeclareMathOperator{\lk}{lk_{\Delta} }

\newcommand{\ffi}{\varphi}
\newcommand{\F}{\mathbb{F}}

\def\eatit#1{}

\newcommand{\myarrow}[2]{\hbox to #1pt{\hfil$\to$\hfil}{\hskip-#1pt{\raise
10pt\hbox to#1pt{\hfil$\scriptscriptstyle #2$\hfil}}}}

\begin{document}

\title{Matroid configurations and symbolic powers of their ideals}

\author[A.V.\ Geramita]{A.V.\ Geramita} 
\address{Department of Mathematics and Statistics, Queen's University, King\-ston, 
Ontario, Canada and Dipartimento di Matematica, Universit\`a di Genova, Genoa, Italy}
\email{Anthony.Geramita@gmail.com}

\author{B.\ Harbourne}
\address{Department of Mathematics\\
University of Nebraska\\
Lincoln, NE 68588-0130 USA}
\email{bharbour@math.unl.edu}

\author{J.\ Migliore} 
\address{Department of Mathematics, 
University of Notre Dame, 
Notre Dame,
IN 46556 \\
USA}
 \email{migliore.1@nd.edu}

\author{U.\ Nagel}
\address{Department of Mathematics\\
University of Kentucky\\
715 Patterson Office Tower\\
Lexington, KY 40506-0027, USA}
\email{uwe.nagel@uky.edu}

\date{July 1, 2015}

\begin{abstract}
Star configurations are certain unions of linear subspaces of
projective space that have been studied extensively. We develop a framework for studying a substantial generalization, which we call  matroid configurations, whose ideals generalize Stanley-Reisner ideals of matroids.   Such a matroid configuration is a  union of complete intersections of a fixed codimension. Relating these to the Stanley-Reisner ideals of matroids and using methods of Liaison Theory  allows us, in particular, to describe the Hilbert function and minimal generators of the ideal of, what we call,  a hypersurface configuration. We also establish that the symbolic powers of the ideal of any matroid configuration are Cohen-Macaulay. As applications, we study ideals coming from certain complete hypergraphs and ideals derived from tetrahedral curves.  We also consider  Waldschmidt constants and  resurgences. In particular, we determine the resurgence of any star configuration and many hypersurface configurations. Previously, the only non-trivial cases for which the resurgence was known were certain monomial ideals and ideals of finite sets of points. Finally, we point out a connection to secant varieties of varieties of reducible forms.
\end{abstract}

\thanks{
{\bf Acknowledgements}: Geramita was partially supported by a Natural Science and Engineering Research Council (Canada) grant.
Harbourne was partially supported by NSA grant number  H98230-13-1-0213.
Migliore was partially supported by NSA grant number H98230-12-1-0204 and by Simons Foundation grant \#309556.
Nagel was partially supported by NSA grant number H98230-12-1-0247 and by Simons Foundation grant \#317096.
We are grateful to Luchezar Avramov for suggesting the proof of Lemma \ref{lucho}.  
We are also grateful to David Cook II for useful comments.
}

\maketitle

\section{Introduction}

Let $k$ be an infinite field and let $R = k[x_0, \ldots , x_n] = \oplus_{i\geq 0} R_i$ be the standard graded polynomial ring.  Suppose $\ell_1, \ldots , \ell_s$ are forms in $R_1$ and consider the hyperplanes defined by them.  Under varying uniformity assumptions on the family of forms (e.g., for some $c \leq n+1$, any subset of $c$ of the linear forms are linearly independent) the collection of codimension $c$ linear subspaces obtained by intersecting subfamilies of these hyperplanes have appeared in the literature, often with motivations found in problems in algebra, geometry or combinatorics (see, e.g., \cite{AS1}, \cite{AS2}, \cite{PSC}, \cite{BH}, \cite{CChG}, \cite{CGVT}, \cite{CHT}, \cite{francisco}, \cite{FMN}, \cite{GHM3},  \cite{GMS},  \cite{MN7}, \cite{PS}, \cite{schwartau},   \cite{TVV}, \cite{V}).  As one can see in these references, the questions asked involve the defining ideal of the collection of such linear spaces, a description of the symbolic powers of those ideals and their finite free resolutions, or more simply questions about the Cohen-Macaulayness and Hilbert function of these ideals.



In this paper we develop a framework for  generalizations of these results.  As we shall see, these generalizations also have interesting consequences in algebra, algebraic geometry and combinatorics.

The generalizations proceed in two steps.  First let  $\lambda = [d_1,\dots,d_s]$ be a vector, where each $d_i$ is a positive integer.  Let $f_1,\dots,f_s$ be homogeneous forms in $R$ with $\deg f_i = d_i$ and let $F_1, \ldots , F_s$ be the hypersurfaces they define in $\PP^n$.  For any $1 \leq c \leq n$ assume that the intersection of any $c+1$ of these hypersurfaces has codimension $c+1$.  We do not require further generality for the $f_i$, not even that they be reduced.    A  {\em $\lambda$-configuration} of {\em codimension $c$} is the union, $V_{\lambda,c}$, of all the codimension $c$ complete intersection subschemes obtained by intersecting $c$ of the hypersurfaces. Notice that any such complete intersection may fail to be irreducible or even reduced. However, it follows from our assumptions that no two of them can have common components. If $\lambda = [1,1,\dots,1]$ then the collection of $V_{\lambda,c}$ includes, what has been called in the literature, {\em codimension $c$ star configurations}.  We will refer to a $V_{\lambda,c}$, for a possibly unspecified $\lambda$ (or $c$),  as a {\em hypersurface configuration (of codimension $c$)}. We will discuss the second step of the generalization when we describe the results of $\S3$.

One purpose of this paper is to show how essentially the same construction as in \cite{GHM3} provides the description of the ideal and the Hilbert function of a $\lambda$-configuration of codimension $c$.  But a new idea is required to show that the Cohen-Macaulay property holds for all symbolic powers of the ideal  of a hypersurface configuration.   

Thus we have the surprising result that these more general configurations of complete intersections (in arbitrary codimension) are just as well-behaved as they are in the case that the components are linear.

The idea of replacing the hyperplanes by hypersurfaces is not new.  For instance: \cite{BH} studies the minimal degrees of generators of the ideals of $\lambda$-configurations when the $f_i$ are general in order to bound Waldschmidt constants; \cite{AS1} describes minimal generators, Hilbert functions and minimal free resolutions  of the configurations $V_{\lambda,c}$ assuming $c=2$ and the $f_i$ are general; \cite{PS} describes the same invariants when the $f_i$ are general and $c$ is arbitrary;  and \cite{AS2} describes the same invariants when $c=2$ and the $f_i$ are replaced by their powers.  Although the title of \cite{AS2} refers to \lq\lq fat" star configurations, the schemes they study are not what have traditionally been referred to as \lq\lq fat" schemes, i.e., schemes defined by symbolic powers.  Consequently our results on symbolic powers (see \S3) do not overlap with the results of \cite{AS2}.
  
The purpose of this note is to put all of these results into a simple framework, and then to illustrate some applications of this method.  To that end, the paper is organized in the following way: in $\S 2$ we set up the basic results we will need.  We show that  $\lambda$-configurations are ACM and find their degrees, Hilbert functions and the minimal generators of their defining ideals.  These results were known but our approach to them (via methods from Liaison Theory) is new.  

In $\S3$ we begin developing a theory of specializing Stanley-Reisner ideals of simplicial complexes. This is the second step of our generalization of linear star configurations.  This section contains the main results of the paper.   We carry out this step  for the class of matroid complexes. We refer to the ideals obtained by replacing the variables of the Stanley-Reisner ideal of these complexes by  homogeneous polynomials as {\em specializations of matroid ideals}. We show that, under certain conditions, these specializations inherit many of the properties of the original matroid ideals. In particular, all their symbolic powers are Cohen-Macaulay. Our results extend most of the earlier investigations for star configurations. Indeed, star configurations and hypersurface configurations are obtained as special cases, namely as specializations of the  Stanley-Reisner ideals of {\em uniform} matroids. 

 The final section  gives applications of our results.  We consider ideals coming from certain complete hypergraphs and ideals derived from tetrahedral curves.  We also discuss connections to Waldschmidt constants and  resurgences;  in particular, we determine the resurgence of any star configuration and many hypersurface configurations. Previously, the only  non-trivial cases for which the resurgence was known were certain monomial ideals and ideals of finite sets of points. 
 
 We also point out a connection to secant varieties of varieties of reducible forms.


\section{The ideal and Hilbert function of a  $\lambda$-configuration of codimension $c$}

Let $R = k[x_0,\dots,x_n] =\oplus_{i\geq 0}R_i$, where $k$ is an arbitrary infinite field, be the standard graded polynomial algebra over $k$.     Recall that a subscheme $V$ of $\PP^n$ is {\em arithmetically Cohen-Macaulay} (ACM) if $R/I_V$ is a Cohen-Macaulay ring, where $I_V$ is the saturated ideal defining $V$.  For a homogeneous  ideal $I \subset R$, the {\em Hilbert function} of $R/I$ is defined by
\[
h_{R/I}(t) = \dim_k [R/I]_t.
\]
When $I = I_V$ is the saturated ideal of a subscheme $V \subset \PP^n$, we usually write $h_V(t)$ for $h_{R/I_V}(t)$. We denote by $\Delta h_{R/I}(t)$ the first difference $h_{R/I}(t) - h_{R/I}(t-1)$, and by $\Delta^2 h_{R/I}(t)$, $\Delta^3 h_{R/I}(t)$, etc. the successive differences.  Suppose that $\delta$ is the Krull dimension of $R/I$.  Then $\Delta^\delta h_{R/I}$ takes on  only finitely many non-zero values.  If we form the vector
$$
(\Delta^\delta h_{R/I}(0), \ldots , \Delta^\delta h_{R/I}(t))
$$
(where the last entry in the vector is the last value of $\Delta^\delta h_{R/I}(t) \neq 0$), then this vector is referred to as the {\em $h$-vector of $R/I$}.  If $I = I_V$ then this vector is called the {\em{$h$-vector of $V$}}.

As in \cite{GHM3}, we will make substantial use of the construction described in the next proposition.  This construction is known in Liaison Theory as {\em Basic Double Linkage} (see \cite[Chapter 4]{KMMNP}). 

\begin{prop} \label{bdgl}
Let $I_C$ be a saturated  ideal defining a  codimension $c$ subscheme $C\subseteq \PP^n$.  
Let $I_S \subset I_C$ be an ideal which defines an ACM subscheme $S$ of codimension $c-1$.  
Let $f$ be a form of degree $d$ which is not a zero-divisor on $R/I_S$.
Consider the ideal $I = f\cdot I_C + I_S$ and let $B$ be the subscheme it defines.

Then $I$ is saturated, hence equal to $I_{B}$, and there is 
an exact sequence
\[
0 \rightarrow I_S(-d) \rightarrow I_C(-d) \oplus I_S \rightarrow I_{B} \rightarrow 0.
\]
In particular, since $S$ is an ACM subscheme of codimension one 
less than $C$, we see that $B$ is an ACM subscheme if and only if $C$ is.  Also, 
\[
\deg B = \deg C + (\deg f)\cdot (\deg S).
\]
Furthermore, let $H_f$ be the 
hypersurface section cut out on $S$ by $f$.  As long as $H_f$ does not  vanish on a component of $C$, we have 
 $B = C \cup H_f$ as schemes.  In any case the Hilbert function 
$h_{B}(t)$ of $R/I_{B}$ is 
\[
h_{B}(t)=h_S(t)-h_S(t-d)+h_C(t-d).
\]
\end{prop}

\begin{rem} In Liaison Theory the scheme $B$ in \ref{bdgl} is often referred to as the {\em basic double link} of $C$ with respect to $f$ and $S$.
\end{rem}

We note that $V$ is an ACM subscheme of codimension zero if and only if $V = \PP^n$ if and only if $I_V = ( 0 )$.  The following is the analogue for $\lambda$-configurations of codimension $c$ of \cite[Proposition 2.9] {GHM3}, which dealt only with {\em linear} star configurations of codimension $c$.

\begin{prop} \label{hf and ideal}
Let $\lambda = [d_1, \dots, d_s]$ and $\mathcal H = \{ F_1,\dots,F_s\}$, where $F_i$ is a hypersurface of degree $d_i$ in $\PP ^n$ (not necessarily reduced), defined by the form $f_i$.  Let $c$ be an integer such that $1 \leq c \leq \min(n,s)$.  Assume that the intersection of any $c+1$ of these hypersurfaces has codimension $c+1$.      Then we have the following facts.

\begin{enumerate}

\item $V_{\lambda,c}$ is ACM.

\item The Hilbert function of $V_{\lambda,c}$ is 
\[
h_{V_{\lambda,c}} (t) = [ h_{V_{\lambda',c-1}}(t) - h_{V_{\lambda',c-1}}(t-d_s) ] + h_{V_{\lambda',c}} (t- d_s).
\]
where $\lambda' = [d_1,\dots,d_{s-1}]$.

\item $\displaystyle \deg V_{\lambda,c} = \sum_{1 \leq i_1 < i_2 < \dots < i_c \leq s} d_{i_1} d_{i_2} \dots d_{i_c}$.

\item The minimal generators of $I_{V_{\lambda,c}}$ are all the products of $s-c+1$ of the forms $f_1,\dots,f_s$. That is, 
\[
I_{V_{\lambda,c}} =  \left  ( \left  \{ f_{i_1} f_{i_2} \cdots f_{i_{s-c+1}} \  | \  {1 \leq i_1 < i_2 < \dots < i_{s-c+1} \leq s} \right \}  \right ).
\]
\end{enumerate}
\end{prop}

\begin{proof}

Nearly the entire proof proceeds exactly as in the proof of \cite [Proposition 2.9]{GHM3}.  As before, the idea is to proceed by induction on $c$ and on $s \geq c$.  For any $c$, if $s = c$ then $V_{\lambda,c}$ is a complete intersection, and all parts are trivial. If $c=1$ and $s$ is arbitrary, then $V_{\lambda,1}$ is the union of $s$ hypersurfaces with no common components, and again all parts are trivial.  Also, (3) is trivial and is included only for completeness.

We now assume that the assertions are true for codimension $c-1$ and all $s$, and for $\lambda$-configurations of codimension $c$ coming from collections of up to $s-1$ hypersurfaces.  Let $\mathcal H' = \{ F_1,\dots,F_{s-1}\}$ and $\lambda^\prime = [d_1,\dots, d_{s-1}]$.  By induction, $V_{\lambda^\prime,c-1}$ and $V_{\lambda^\prime,c}$ are both ACM and the ideals are of the stated form.  Furthermore, $f_s$ is not a zero divisor on $R/I_{V_{\lambda^\prime,c-1}}$.  By Proposition \ref{bdgl}, 
\[
I_{V_{\lambda,c}} = f_s \cdot I_{V_{\lambda^\prime,c}} + I_{V_{\lambda^\prime,c-1}},
\]
and $V_{\lambda,c}$ is ACM.  This is (1).  Statements (2) and (4) also follow immediately from this construction of the ideal, again using induction and Proposition \ref{bdgl}.      
\end{proof}

\noindent We note that (4) was shown in \cite{PS}.

\begin{cor} \label{hvect}
Let $\lambda = [d_1,\dots,d_s]$ and $\lambda' = [d_1,\dots,d_{s-1}]$.  Then for any $i \geq 1$ we have
\[
\Delta^i h_{V_{\lambda,c}} (t) = [ \Delta^i h_{V_{\lambda',c-1}}(t)  - \Delta^i h_{V_{\lambda',c-1}}(t - d_s) ]
+ \Delta^i h_{V_{\lambda',c}}(t - d_s).
\]
In particular, let $X$ be the hypersurface section of $V_{\lambda',c-1}$ by $F_s$.  Let $\underline{h}_{V_{\lambda',c}}$ be the $h$-vector of $V_{\lambda',c}$,  $\underline{h}_{V_{\lambda,c}}$  the $h$-vector of $V_{\lambda,c}$, and   $\underline{h}_{X}$  the $h$-vector of $X$.  Then
\[
\underline{h}_{V_{\lambda,c}} = \underline{h}_X + \underline{h}_{V_{\lambda',c}}(-d_s).
\]

\end{cor}

\begin{proof}
The first part is immediate and the second part comes from taking $i = n-c+1$, and remembering that $\dim V_{\lambda,c} = \dim V_{\lambda',c} = n-c$, while $\dim V_{\lambda',c-1} = n-c+1$.
\end{proof}

\begin{example}
We illustrate the $h$-vector computation from Corollary \ref{hvect}.  Suppose $n = 3$,  $\lambda = [4,3,3,2]$, and consider $c=2$ and $c=3$.  Let us compute the $h$-vectors.   We first find the $h$-vectors of the successive codimension 2 hypersurface configurations in $\PP^3$.  The integer in column $t$ (starting with $t = 0$) represents the value of the $h$-vector in degree $t$.

The first scheme, $V_{(4,3),2}$, is a complete intersection of degree 12, so the $h$-vector is well known:
\[
\begin{array}{rllllllllllllll}
V_{(4,3),2} : & 1 & 2 & 3 & 3 & 2 & 1  
\end{array}
\]
To compute the $h$-vector of $V_{(4,3,3),2}$, note that $V_{(4,3),1}$  is a hypersurface of degree $4+3 = 7$, and we are cutting it with a hypersurface of degree 3 to obtain $X$. Thus we have the following $h$-vector computation:
\[
\begin{array}{rrrrrrrrrrrrrrrrrr}
 & V_{(4,3),2} (-3): & - & - & - & 1 & 2 & 3 & 3 & 2  & 1  \\
 & {X}: & 1 & 2 & 3 & 3 & 3 & 3 & 3 & 2 & 1  &  \\ \hline
&V_{(4,3,3),2}: & 1 & 2 & 3 & 4 & 5 & 6 & 6 & 4 & 2 
\end{array}
\]
Next, we compute the $h$-vector of $V_{(4,3,3,2),2}$. Now $X$ is the complete intersection of a hypersurface of degree $4+3+3=10$ and one of degree 2.
\[
\begin{array}{rrrrrrrrrrrrrrrrrr}
 & V_{(4,3,3),2} (-2): & - & - & 1 & 2 & 3 & 4 & 5 & 6 & 6 & 4 & 2  \\
 & X : &  1 & 2 & 2 & 2 & 2 & 2 & 2 & 2 & 2 & 2 & 1 \\ \hline
&V_{(4,3,3,2),2}: &1 & 2 & 3 & 4 & 5 & 6 & 7 & 8 & 8 & 6 & 3
\end{array}
\]
We now turn to the $h$-vectors  of  codimension 3 hypersurface configurations.  The first, $V_{(4,3,3),3}$, is again a complete intersection, so its $h$-vector is
\[
\begin{array}{rlllllllllllllll}
V_{(4,3,3),3}: & 1 & 3 & 6  & 8 & 8 & 6 & 3 & 1.
\end{array}
\]
Now $X$ is the hypersurface section of $V_{(4,3,3),2}$ by $F_4$, which has degree 2.  To compute the $h$-vector of $X$ we first ``integrate'' the $h$-vector of $V_{(4,3,3),2}$, and then we take a shifted difference:
\[
\begin{array}{rrrrrrrrrrrrrrrrrr}
& 1 & 3 & 6 & 10 & 15 & 21 & 27 & 31 & 33 & 33 & 33 & 33 &  \cdots   \\
& - & - & 1 & 3 & 6 & 10 & 15 & 21 & 27 & 31 & 33 & 33 & \cdots \\ \hline
X: & 1 & 3 & 5 & 7 & 9 & 11 & 12 & 10 & 6 & 2
\end{array}
\]
Finally, we apply Corollary \ref{hvect}:
\[
\begin{array}{rrrrrrrrrrrrrrrrrr}
 & V_{(4,3,3),3} (-2): & - & - &  1 & 3 & 6  & 8 & 8 & 6 & 3 & 1 \\
 & X : &  1 & 3 & 5 & 7 & 9 & 11 & 12 & 10 & 6 & 2 \\ \hline
&& 1 & 3 & 6 & 10 & 15 & 19 & 20 & 16 & 9 & 3 
\end{array}
\]

\end{example}


\section{Specializations of Matroid Ideals and their Symbolic Powers}

We begin with a lemma, whose proof was suggested to us by L. Avramov. 

\begin{lem} \label{lucho}
Let $S = k[y_1,\ldots,y_s]$ and  $R = k[x_0,...,x_n]$ be polynomial rings over
a field $k$. Let $f_1,\dots,f_s \in R$ be an $R$-regular sequence of homogeneous
elements of the same degree (with respect to the standard grading). Let $I = (g_1,...,g_r)$
be a homogeneous ideal in $S$, so each $g_i$ is a polynomial $g_i = g_i
(y_1,...,y_s)$. Let $p_i = g_i(f_1,...,f_s)$ and let $J = (p_1,...,p_r)$.
Then $I$ and $J$ have the same graded Betti numbers over $S$ and $R$, respectively, except possibly with shifts
which depend on the degrees of the $f_i$.  In particular, $S/I$ is Cohen-Macaulay
if and only if $R/J$ is Cohen-Macaulay.

If $I$ is a monomial ideal then we can drop the requirement that the $f_i$ all have the same degree.
\end{lem}

\begin{proof}
We require the $f_i$ to have the same degree in order that the $g_i$ continue to be homogeneous, and also so that the maps in the minimal free resolution continue to be graded.  When $I$ is monomial, this restriction is not needed.  

Define a homomorphism of $k$-algebras  $\ffi : S \rightarrow R$ by $\ffi(y_i) = f_i$ for  $i=1,\dots,s$.  
It is flat because the $f_i$ form a regular sequence.  Let $\FF$ be a graded 
minimal free  resolution of $S/I$ over $S$. Then $R\otimes_S \FF$ is a graded minimal
free resolution of $R/J$ over $R$.  
\end{proof}

\begin{example}
Take  $\deg(f_i) = 2$ for all $i$ and suppose the Betti diagram for $R/I$ is the one on the left in Figure \ref{SpecFig}.   Then the Betti diagram for $R/J$ is the one on the right in Figure \ref{SpecFig}.

{ \begin{figure}[htbp]
\begin{center}
{
\hbox{\hspace{1in}\vbox{
\begin{verbatim}

               0  1  2 3             0  1  2 3
        total: 1 10 12 3      total: 1 10 12 3
            0: 1  .  . .          0: 1  .  . .
            1: .  .  . .          1: .  .  . .
                  ...                   ...
            7: .  .  . .         16: .  .  . .
            8: .  4  . .         17: .  4  . .
            9: .  3  6 .         18: .  .  . .
           10: .  2  4 2         19: .  3  . .
           11: .  1  2 1         20: .  .  6 .
                                 21: .  2  . .
                                 22: .  .  4 .
                                 23: .  1  . 2
                                 24: .  .  2 .
                                 25: .  .  . 1
\end{verbatim}
}}}
\caption{Comparing a Betti diagram with that of a specialization.}
\label{SpecFig}
\end{center}
\end{figure}

}

\end{example}

We now recall a few concepts for simplicial complexes. 
A {\em matroid} $\Delta$  on a vertex set $[s] = \{1,2,\ldots,s\}$ is a non-empty collection of subsets of $[s]$ that is closed under inclusion and satisfies the following property: If $F, G$ are in $\Delta$ and $|F| > |G|$, then there is some $j \in F$ such that $G \cup \{j\}$ is in $\Delta$. We will always consider $\Delta$ as a simplicial complex. Equivalently, a matroid is a simplicial complex $\Delta$ such that, for every subset $F \subseteq [s]$, the restriction $\Delta|_F = \{G \in \Delta \; | \; G \subseteq F\}$ is pure, that is, all its facets have the same dimension.  

  For a subset $F
\subseteq [s]$, we
write $y_F$ for the squarefree monomial $\prod_{i \in F} y_i$.
The {\em Stanley-Reisner ideal} of $\Delta$ is
$I_\Delta =(y_F \ | \ F \subseteq [s],\ F\not\in \Delta)$
and the corresponding {\em Stanley-Reisner ring} is $k[\Delta]=S/I_\Delta$, where $S = k[ y_1,\ldots,y_s]$. It is Cohen-Macaulay. In fact, it was shown in \cite[Theorem 3.3]{NRo} that matroid complexes are, what is referred to there as {\em squarefree glicci simplicial complexes} (see \cite{NRo} for the definition). We now explain this result in a more detailed way. 

Let $\Delta$ be any simplicial complex on $[s]$. Each subset  $F \subseteq [s]$ induces the  following simplicial subcomplexes of $\Delta$:
the {\em link of $F$}
$$
\lk F = \{G \in \Delta \  | \ F \cup G \in \Delta, F \cap G = \emptyset
\},
$$
and the {\em deletion}
$$
\Delta_{-F} = \{G \in \Delta \ | \ F \cap G = \emptyset  \}.
$$
For each vertex $j$ of $\Delta$, the link $\lk j$ and the deletion $\Delta_{-j}$ are simplicial complexes on $[s] \setminus \{j\}$. Moreover, if $\Delta$ is a matroid, then  $\lk j$ and  $\Delta_{-j}$ are again matroids (see, e.g., \cite{Oxley-book-92}), where $\dim S/I_{\Delta_{-j}} S = \dim S/I_{\Delta} + 1$.  Furthermore $y_j$ is not a zerodivisor on $S/ I_{\Delta_{-j}}$ and (see \cite[Remark 2.4]{NRo}) 
\begin{equation}
    \label{eq:squarefree BDL}
I_{\Delta} = y_j I_{\lk j} S +  I_{\Delta_{-j}} S. 
\end{equation} 
It follows that $I_{\Delta}$ is a basic double link of $I_{\lk j} S$ with respect to $y_j$ and $I_{\Delta_{-j}}S$, as in Proposition \ref{bdgl}.   Replacing $I_{\Delta}$ by $I_{\lk j}$, and iterating, one sees that $I_{\Delta}$ can be obtained from a complete intersection generated by variables via  a series of basic double links through squarefree monomial ideals. This means that $I_{\Delta}$ is squarefree glicci. 

We use these facts to establish the following result. 

\begin{thm}
    \label{thm:matroid spec}
Let $\Delta$ be a matroid on $[s]$  of dimension $s-1 -c$, and let $P_1,\ldots,P_t$ be the associated prime ideals of $k[\Delta]$.  Assume $n \geq c$ and that  $f_1,\ldots,f_s \in R = k[x_0,\ldots,x_n]$ are homogeneous polynomials such that any subset of at most $c+1$ of them forms an $R$-regular sequence. Consider the ring homomorphism 
\[
\ffi: S = k[y_1,\ldots,y_s] \to R, \; y_i \mapsto f_i. 
\]
If $I$ is an ideal of $S$ we will write $\ffi_*(I)$ to denote the ideal in $R$ generated by $\ffi(I)$.
Then the following facts are true. 
\begin{enumerate}

\item The ideal $\ffi_* (I_{\Delta})$ is a Cohen-Macaulay ideal  of codimension $c$. 

\item ${\displaystyle  \ffi_* (I_{\Delta}) = \bigcap_{i=1}^t \ffi_* (P_i)}$. 

\item If $\F_{k[\Delta]}$ is a graded minimal free resolution of $k[\Delta]$ over $S$, then $\F_{k[\Delta]} \otimes_S R$ is a graded minimal free resolution of $R/\ffi_* (I_{\Delta})$ over $R$. 

\end{enumerate} 
\end{thm}    

The ideal $\ffi_* (I_{\Delta})$ is said to be obtained by {\em specialization} from the matroid ideal $I_{\Delta}$.  The subscheme of $\PP^n$ defined by $\ffi_* (I_{\Delta})$ is called a {\em matroid configuration}.

\begin{proof} 
We begin by showing the first two claims. We use induction on $c \ge 1$. If $c = 1$, then $I_{\Delta}$ is a principal ideal, and the assertions are clearly true. 

Let $c \ge 2$. Now we use induction on $s \ge c$. If $s = c$, then $I_{\Delta}$ is a complete intersection, and again the claims are clear. 

Let $s > c$.   As pointed out above, $\lk s$ and $\Delta_{-s}$ are matroids on $[s-1]$, and their Stanley-Reisner ideals have codimensions $c$ and $c-1$, respectively. Thus claims (1) and (2) hold true for these ideals by the induction hypothesis. The assumption on the forms  $f_i$ gives 
\[
\ffi_* (I_{\lk s} S)  : f_s = \ffi_* (I_{\lk s} S). 
\]
Moreover, Relation \eqref{eq:squarefree BDL} yields 
\[
\ffi_* (I_{\Delta}) = f_s \ffi_*( I_{\lk s} S) +  \ffi_*( I_{\Delta_{-j}} S). 
\]
Hence $\ffi_* (I_{\Delta})$ is a basic double link of the Cohen-Macaulay ideal $\ffi_*( I_{\lk s} S)$, and thus it is Cohen-Macaulay of codimension $c$, proving (1). 

To show (2), denote by $P_1,\ldots,P_u$ the associated prime ideals of $k[\Delta]$ that do not contain $y_s$. For $j = u+1,\ldots,t$, define monomial prime ideals $P'_j$  generated by variables in $\{y_1,\ldots,y_{s-1}\}$   by $P_j = y_s R + P_j'$. Then 
\[
I_{\lk s} S = \bigcap_{j=1}^u P_j \quad \text{ and } \quad  I_{\Delta_{-j}} S =  \bigcap_{j= u+1}^t P'_j. 
\]
Moreover, since $I_{\Delta}$ is squarefree, we have 
\[
I_{\Delta} = I_{\lk s} S \cap (y_s,  I_{\Delta_{-j}}) S. 
\]
Applying the homomorphism $\ffi$ we obtain
\[
\ffi_* (I_{\Delta}) \subset \ffi_* (I_{\lk s} S) \cap \ffi_* ((y_s,  I_{\Delta_{-j}}) S) \subset \bigcap_{j=1}^u \ffi_* (P_j) \ \cap \  \bigcap_{j= u+1}^t (f_s, \ffi_*(P'_j)) =  \bigcap_{j=1}^t \ffi_* (P_j).  
\]
Notice that the ideal on the right-hand side is unmixed and  has degree $\deg (I_{\lk s}) + \deg (f_s) \cdot \deg (I_{\lk s})$. Since $I_{\Delta}$ is also an unmixed ideal with the same degree, the two ideals must be equal, completing the argument for (2). 

Finally, we show Claim (3). Let us say that the polynomials $f_i$ satisfy property $(P_m)$ if any subset if at most $m+1$ of them forms an $R$-regular sequence. If the $f_i$ satisfy property $(P_{s})$, then Claim (3) is a consequence of Lemma \ref{lucho}. 

Let $s > c+1$. We use induction on the difference between $s$ and the number $c+1$, as determined by the assumption on the forms $f_i$.  The idea is to replace the given forms $f_i$ by new forms, satisfying a stronger condition. 
By induction, we know that Claim (3) is true if we substitute for the $y_i$ forms in any polynomial ring  such that any subset of at most $c+2$ of these polynomials forms a regular sequence.  
So let $z$ be a new variable and define a polynomial ring $T = R[z]$. For each $i \in [s]$, let $f'_i \in (f_i, z) T$ be a general polynomial of degree $\deg f_i$. Now consider the homomorphism 
\[
\gamma: S \to T, \ y_i \mapsto f'_i. 
\]
Observe that any subset of at most $c+2$ of the polynomials $f'_i$ forms  a $T$-regular sequence. Thus, 
the induction hypothesis gives that $\F_{k[\Delta]} \otimes_S T$ is a graded minimal free resolution of $T/\gamma_* (I_{\Delta})$ over $T$. We have the following graded isomorphism
\[
T/(\gamma_* (I_{\Delta}), z) \cong R/\ffi_* (I_{\Delta}). 
\]
Since $T/\gamma_* (I_{\Delta})$ is Cohen-Macaulay and $\dim T/\gamma_* (I_{\Delta}) = 1 + \dim R/\ffi_* (I_{\Delta}) $, $z$ is not a zerodivisor of $T/\gamma_* (I_{\Delta})$. It follows that $\F_{k[\Delta]} \otimes_S T \otimes_T T/zT \cong \F_{k[\Delta]} \otimes_S R$ is a graded minimal free resolution of $k[\Delta]$ over $R$, as claimed. 
\end{proof}

\begin{example}
     \label{exa:d-stars}
Consider the ideal of $S$
\[
I_{s, c} = \bigcap_{1 \le i_1 < i_2 < \cdots < i_c \le s} (y_{i_1}, y_{i_2},\ldots,y_{i_c}), 
\]
generated by all products of $s-c+1$ distinct variables in $\{y_1,\ldots,y_s\}$.    It is the Stanley-Reisner ideal of a uniform matroid on $[s]$ whose facets are all the cardinality $s-c$ subsets of $[s]$. Hence, with the hypotheses of  Theorem \ref{thm:matroid spec},  every   specialization $\ffi_* (I_{s, c})$ is again Cohen-Macaulay of codimension $c$ and 
\[
\ffi_* (I_{s, c}) = \bigcap_{1 \le i_1 < i_2 < \cdots < i_c \le s} (f_{i_1}, f_{i_2},\ldots,f_{i_c}). 
\]
Note that $\ffi_* (I_{s, c})$ is the ideal of a hypersurface configuration in $\PP^n$ and that any hypersurface configuration arises in this way. 
\end{example}

Observe that the Alexander dual of $I_{s, c}$ is $I_{s, s-c+1}$. Since each is the dual of the other and both are Cohen-Macaulay, both ideals have a linear free resolution (see \cite{ER}). This also follows from the fact that $I_{s, c}$ is a squarefree strongly stable ideal. Extending results in \cite{CN}, it was shown in \cite{NRe} that all squarefree strongly stable ideals that are generated in one degree have a linear cellular minimal free resolution that can be explicitly described using a complex of boxes. It turns out that in the case of the ideal $I_{s,c}$, this complex of boxes can be realized as a subdivision of a simplex on $c$ vertices. 

Now, applying Theorem \ref{thm:matroid spec}, we obtain the following result about hypersurface configurations.  

\begin{cor}
 Each specialization $\ffi_* (I_{s, c})$ admits an explicit graded minimal free resolution, including a description of the maps, that stems from a cellular resolution of $I_{s, c}$. 
\end{cor}

The graded Betti numbers in the resolution of $\ffi_* (I_{s, c})$ (but not the maps) have been determined in \cite{PS}. 

Theorem \ref{thm:matroid spec} can be extended to symbolic powers.

\begin{thm}
    \label{thm:symb powers}
Adopt the notation and assumptions of Theorem \ref{thm:matroid spec}. Then the following facts are true for each positive integer $m$:
\begin{enumerate}

\item  ${\displaystyle  \ffi_* (I_{\Delta})^{(m)}  =    \bigcap_{i=1}^t \ffi_* (P_i)^m  } = \ffi_* (I_{\Delta}^{(m)})$. 

In particular, $\ffi_* (I_{\Delta})^{(m)}$ is generated by monomials in the $f_i$ and has codimension $c$. 

\item If $\F$ is a graded minimal free resolution of $R/I_{\Delta}^{(m)}$ over $S$, then $\F \otimes_S R$ is a graded minimal free resolution of $R/\ffi_* (I_{\Delta})^{(m)}$ over $R$. In particular, $R/ \ffi_* (I_{\Delta})^{(m)}$ is Cohen-Macaulay. 

\end{enumerate} 
\end{thm}    

\begin{proof} 
We begin by showing $\ffi_* (I_{\Delta})^{(m)} =    \bigcap_{i=1}^t \ffi_* (P_i)^m$. 
The assumption on the polynomials $f_i$ and Theorem \ref{thm:matroid spec}(2) give that a prime ideal $P$ of $R$ is an associated prime of $R/\ffi_* (I_{\Delta})$ if and only if $P$ is an associated prime ideal of  $R/\ffi_* (P_i)$ for exactly one $i \in [t]$.  Using that $\ffi_* (P_i)$ is a complete intersection, and so $\ffi_* (P_i)^m$ is Cohen-Macaulay, we get $\ffi_* (I_{\Delta})^m R_P = \ffi_* (P_i)^m R_P$. This implies $\ffi_* (I_{\Delta})^{(m)} =    \bigcap_{i=1}^t \ffi_* (P_i)^m$, as desired. 

Assume now that $s \le c+1$. It was shown independently in \cite{V} and \cite{TT} that, for each positive integer $m$, the ideal 
\[
I_{\Delta}^{(m)} = \bigcap_{j=1}^t P_j^m 
\]
is Cohen-Macaulay. Hence Lemma \ref{lucho} gives that $\ffi_* (I_{\Delta}^{(m)})$ is Cohen-Macaulay  and that its resolution can be obtained from the resolution of $S/I_{\Delta}^{(m)}$ over $S$. 
Recall that in the case $s \le c+1$, the homomorphism $\ffi$ is flat. Thus, using the identity  established above and \cite[Theorem 7.4(ii)]{M},  we get
\[
\ffi_* (I_{\Delta}^{(m)}) = \bigcap_{j=1}^t \ffi_* (P_j^m) = \bigcap_{j=1}^t (\ffi_* (P_j))^m =  \ffi_* (I_{\Delta})^{(m)}.
 \]
 
Let $s > c+1$. We use induction on the difference between $s$ and the number $c+1$ as in the proof of Theorem \ref{thm:matroid spec} to show the remaining claims. Adopt the notation employed in the proof of Theorem \ref{thm:matroid spec}. The induction hypothesis gives that 
\[
\gamma_* (I_{\Delta}^{(m)} )  =    \bigcap_{i=1}^t \gamma_* (P_i)^m  
\]
is Cohen-Macaulay. By the choice of the $f'_i$, the variable $z$ is not a zerodivisor of any $T/\gamma_* (P_i)$. Hence, all the ideals $(z, \gamma_* (I_{\Delta}^{(m)} ))$ and $(z, \gamma_* (P_i)^m)$ are Cohen-Macaulay, and 
\[
(z, \gamma_* (I_{\Delta}^{(m)} )) \subset  \bigcap_{i=1}^t  (z, \gamma_* (P_i)^m). 
\]
Since both ideals are unmixed of codimension $c+1$ and have the same degree, they must be equal. It follows that 
\[
\ffi_* (I_{\Delta}^{(m)}) =    \bigcap_{i=1}^t \ffi_* (P_i)^m,
\]
as desired. 

Finally, using the isomorphism $T/(\gamma_* (I_{\Delta}^{(m)}), z) \cong R/\ffi_* (I_{\Delta}^{(m)})$ and the fact that $z$ is not a zerodivisor  of $T/\gamma_* (I_{\Delta}^{(m)})$ gives Claim (3). 
\end{proof}

The above result applies to $\lambda$-configurations. 

\begin{cor} 
    \label{symb power stars}
Let $\lambda = [d_1, \dots, d_s]$ and $\mathcal H = \{ F_1,\dots,F_s\}$, where $F_i$ is a hypersurface of degree $d_i$ in $\PP ^n$ (not necessarily reduced), defined by the form $f_i$.  Let $c$ be an integer such that $1 \leq c \leq \min(n,s)$.  Assume that the intersection of any $c+1$ of these hypersurfaces has codimension $c+1$.  
Let  $V_{\lambda,c}$ be the corresponding $\lambda$-configuration in codimension $c$.  Then every symbolic power of $I_{V_{\lambda,c}}$ is Cohen-Macaulay.  Furthermore, the minimal generators of each $I_{V_{\lambda,c}}^{(m)}$ are  monomials in the $f_i$.
\end{cor}

\begin{proof}
As pointed out in Example \ref{exa:d-stars}, the ideal $I_{s,c}$ is the Stanley-Reisner ideal of a uniform matroid. Hence Theorem \ref{thm:symb powers}, gives that,  for each positive integer $m$, 
\[
I_{V_{\lambda,c}}^{(m)} = \ffi_* (I_{s, c})^{(m)} = \ffi_* (I_{s,c}^{(m)}) =  \bigcap_{1 \le i_1 < i_2 < \cdots < i_c \le s} (f_{i_1}, f_{i_2},\ldots,f_{i_c})^m 
\]
is Cohen-Macaulay of codimension $c$. 
\end{proof}

In the special case, where all the forms $f_i$ are linear, the ideal $\ffi_* (I_{s,c})^{(m)}$ is a symbolic power of the ideal of a star configuration. For this case, Corollary \ref{symb power stars} has been shown previously in \cite{GHM3}. 

For an application in the next section we note the following result.

\begin{prop}
   \label{prop:inclusion-preserving} 
Let $\Delta$ be a matroid on $[s]$  of dimension $s-1-c$.   Assume $n \geq c$ and that  $f_1,\ldots,f_s \in R = k[x_0,\ldots,x_n]$ are homogeneous polynomials such that any subset of at most $c+1$ of them forms an $R$-regular sequence.   Consider the ring homomorphism $\ffi: S = k[y_1,\ldots,y_s] \to R$, defined by $y_i \mapsto f_i$. Whenever $m$ and $r$ are positive integers, we have the following facts:

\begin{enumerate}

\item $I_{\Delta}^{(m)} \subseteq I_{\Delta}^r$ implies  $\ffi_*( I_{\Delta})^{(m)} \subseteq \ffi_* (I_{\Delta})^r$.

\item If $f_1,\ldots,f_s$ is an $R$-regular sequence, then 
\[
I_{\Delta}^{(m)} \subseteq I_{\Delta}^r \text{ if and only if } \ffi_*( I_{\Delta})^{(m)} \subseteq \ffi_* (I_{\Delta})^r.
\]
\end{enumerate}
\end{prop} 

\begin{proof} 
Assume first $I_{\Delta}^{(m)} \subseteq I_{\Delta}^r$. 
Using Theorem \ref{thm:symb powers}, we get
\[
\ffi_*( I_{\Delta})^{(m)} = \ffi_*( I_{\Delta}^{(m)}) \subseteq \ffi_* ( I_{\Delta}^r)  =  \ffi_* (I_{\Delta})^r . 
\]

To show the second claim, it remains to show the reverse implication. Our assumption on the $f_i$ gives that $\ffi$ is a faithfully flat homomorphism. Hence $I_{\Delta}^{(m)} \nsubseteq I_{\Delta}^r$ implies $\ffi_*(I_{\Delta}^{(m)}) \nsubseteq \ffi_*(I_{\Delta}^r)$, and we are done. 
\end{proof}


\section{Applications}

Our first application will be to construct an ideal coming in a natural way from a multipartite hypergraph, and recognize it as also coming from our construction. Thus it and its symbolic powers will be Cohen-Macaulay.  Its minimal free resolution will also be known.

Let $G$ be a $c$-uniform complete multipartite hypergraph. More precisely, following \cite [Definition 3.4]{NRe}, we will assume that it is a complete {\em $s$-partite} hypergraph, $s \ge c$,  on a partitioned vertex set $X^{(1)} \sqcup \dots \sqcup X^{(s)}$, consisting of all $c$ element subsets with each element  coming from a different $X^{(i)}$.  Let $|X^i| =e_i$.  By \cite[Theorem 3.13]{NRe}, the ideal $I_G$ has a linear resolution.  Thus, the Alexander dual, $I_G^{\vee}$,  of the ideal $I_G$ of $G$ is Cohen-Macaulay.

By definition,
\[
I_G^{\vee} = \bigcap_{1 \leq i_1 < i_2 < \dots < i_c \leq s}  \ \bigcap_{
k = 1,\ldots,c \atop  1 \le j_k \le e_k } (x_{i_1,j_1} , x_{i_2, j_2} , \dots , x_{i_c, j_c}), 
\]  
where each variable $x_{i,j}$ corresponds to the vertex $v_{i, j}$ in $X^{(i)}$. 

We will now specialize this ideal by assigning to each variable $x_{i, j}$  a homogenous polynomial $A_{i,j}$. Thus,  to each face of $G$
\[
\{ v_{i_1,j_1},\dots,v_{i_{c},j_{c}} \}
\]
we can associate an ideal of the form $ (A_{i_1,j_1}, \dots, A_{i_{c},j_{c}})$.  We will focus on the intersection of all such ideals, assuming that the $A_{i,j}$ meet properly.

More formally, let $R = k[x_0,\dots,x_n]$.  Consider sets of homogeneous polynomials in $R$ 
\[
\begin{array}{rcl}
\mathcal A_1 & = & \{ A_{1,1}, A_{1,2},  \dots, A_{1,e_1} \} \\

\mathcal A_2 & = & \{ A_{2,1}, A_{2,2}, \dots, A_{2,e_2} \} \\

& \vdots \\

\mathcal A_s& = & \{ A_{s,1}, A_{s,2}, \dots,  A_{s,e_s} \} \\
\end{array}
\]
where we assume that any choice of $n+1$ of them is a regular sequence, where we choose at most one $A_{i,j}$ from each subset.  Now choose any codimension $1 \leq c \leq n$.  We define a scheme $W_c$ by constructing the following saturated ideal:
\[
I_{W_c} = \bigcap_{1 \leq i_1 < i_2 < \dots < i_c \leq s}  \ \bigcap_{
k = 1,\ldots,c \atop  1 \le j_k \le e_k } ( A_{i_1,j_1} , A_{i_2, j_2} , \dots , A_{i_c, j_c}). 
\]  
That is, thinking of the $A_{i,j}$ as hypersurfaces, we form all possible codimension $c$ complete intersections such that no two generators within a complete intersection come from the same $\mathcal A_i$.  Since any choice of $n+1$ of the $A_{i,j}$ form a regular sequence, no two of these codimension $c$ complete intersections have any common components.  Hence the above construction gives the saturated ideal of an unmixed codimension $c$ subscheme of $\PP^n$, which we  call $W_c$.  

Notice  that if $e_1 = \dots = e_s = 1$ then we have a $\lambda$-configuration of codimension $c$ (where $\lambda = [\deg A_{1,1}, \deg A_{2,1}, \dots, \deg A_{s,1}]$).  If furthermore all the $A_{i,j}$ have degree 1 then we have a linear star configuration of codimension $c$.  

\begin{cor}
The saturated ideal $I_{W_c}$ can be realized as the ideal of a suitable $\lambda$-configuration.  Hence its Hilbert function can be computed,  all its symbolic powers are Cohen-Macaulay, and its minimal free resolution can be described as above.
\end{cor}

\begin{proof}
For $1 \leq i \leq s$ let $f_i = \prod_{j = 1}^{e_i} A_{i,j}$.  Then clearly $W_c$ is the $\lambda$-configuration of codimension $c$ associated to $\{ f_1, f_2, \dots, f_s \}$.  Thus the above ideal can be obtained by specialization, so the assertions follow from our earlier results.
\end{proof}

For a second application of our methods, note that the $m$-th symbolic power of the ideal of a $\lambda$-configuration is the intersection of the $m$-th powers of the complete intersections that go into its construction (see for instance Theorem \ref{thm:symb powers} (1)), regardless of whether these complete intersections are reduced or irreducible. (For instance, the $m$-th symbolic power of the ideal $I_{W_c}$ constructed above is the intersection of the ideals $( A_{i_1,j_1}, \dots, A_{i_c, j_c} )^m$.)   We have seen that all such symbolic powers are Cohen-Macaulay.  

By a slight abuse of notation, we will refer to these complete intersections as the components of the $\lambda$-configuration.  One can ask about properties of  the ideal formed by allowing the powers of the ideals of components to be different.  Not much is known about this problem except in the case of fat points in $\PP^2$ \cite[Example 4.2.2]{CHT} and of tetrahedral curves in $\PP^3$.  The latter are subschemes in $\PP^3$ defined by ideals of the form
\begin{equation}\label{tetra}
( x_0,x_1 )^{p_1} \cap ( x_0,x_2 )^{p_2} \cap ( x_0,x_3 )^{p_3} \cap ( x_1,x_2 )^{p_4} \cap ( x_1,x_3 )^{p_5} \cap ( x_2,x_3 )^{p_6} .
\end{equation}
In this case, combining the work in \cite{schwartau}, \cite{MN7}, \cite{FMN} and \cite{francisco}, much is known about the ideal, the minimal free resolution, the deficiency module and the even liaison class of such curves.  A broad array of heavy machinery, largely based on the fact that these are monomial ideals, went into the results in these papers.  

In \cite [Remark 7.3]{MN7}, it was observed that if we replace the indeterminates $x_0,x_1,x_2,x_3$ by a regular sequence $f_1,f_2, f_3, f_4$, then most of the results in \cite{MN7} continue to hold in $\PP^3$.  The argument was that the liaison approach used therein can be extended to this setting.  In \cite [Question 7.4 (7)]{MN7} it was asked whether the same sort of program can be carried out in higher-dimensional projective space, and it was noted that now issues of local Cohen-Macaulayness will arise, even in the codimension two case.  

Our observation now is that all of these results can be extended almost immediately to higher-dimensional  projective space using Lemma \ref{lucho}.  For instance, we have the following generalization of the main theorem of \cite{francisco}, which built on the work in  \cite{schwartau}, \cite{MN7}, \cite{FMN}.

\begin{cor}
Let $f_1,f_2,f_3,f_4$ be a regular sequence of homogeneous polynomials in $k[x_0,\dots,x_n]$.  Let $C$ be the codimension two scheme defined by the saturated ideal
\[
( f_1,f_2 )^{p_1} \cap ( f_1, f_3 )^{p_2} \cap ( f_1,f_4 )^{p_3} \cap ( f_2,f_3 )^{p_4} \cap ( f_2,f_4 )^{p_5} \cap ( f_3,f_4 )^{p_6} .
\]
Assume without loss of generality that $p_1+p_6 = \max \{ p_1+p_6, p_2+p_5, p_3+p_4 \}$.  Then $C$ is ACM if and only if at least one of the following conditions holds:

\begin{itemize}
\item[(i)] $p_1=0$ or $p_6 = 0$.
\item[(ii)] $p_1+p_6 = \varepsilon + \max \{ p_2+p_5,p_3+p_4\}$, where $\varepsilon \in \{0,1\}$.
\item[(iii)] $2p_1 < p_2+p_3+3-p_6$ or $2p_1 < p_4+p_5+3-p_6$ or $2p_6 < p_2+p_4+3-p_1$ or $2p_6 < p_3+p_5+3-p_1$.
\item[(iv)] All inequalities of (iii) fail, $p_1+p_6 = 2+p_2+p_5=2+p_3+p_4$, and $p_1+p_3+p_5$ is even.
\end{itemize}
\end{cor}

 \begin{proof}  
If $f_1 = x_0, \ f_2 = x_1, \ f_3 = x_2, \ f_4 = x_3$, and $n=3$, then this is the result of \cite{francisco} taken verbatim. Call the corresponding ideal $J$.

Now replace each $x_i$ by $f_i$ in the monomials generating $J$ and denote by $I$ the ideal in $R = k[x_0,\dots,x_n]$ generated by these monomials in the $f_i$. Using again that the substitution homomorphism is flat by the assumption on the $f_i$, we see that $J$ is equal to the ideal considered in the statement. Moreover, Lemma \ref{lucho} gives that the length of its resolution over $R$ is the same as the length of the resolution of $I$ over $k[x_0,\ldots,x_3]$,  which concludes the argument.
\end{proof}

As a third application of our results we consider how they can 
be used to calculate Waldschmidt constants and resurgences.  Let $(0)\neq I\subsetneq R=k[x_0,\ldots,x_n]$ be a homogeneous ideal.
We denote by $\alpha(I)$ the least degree among nonzero forms in $I$.
The {\em Waldschmidt constant} $\widehat{\alpha}(I)$ of $I$ is
$$\widehat{\alpha}(I) = \lim_{m\to\infty}\frac{\alpha(I^{(m)})}{m}.$$
This limit is known to exist and in various situations
it is of interest to compute it or at least to estimate it
(\cite{BH, GHvT, DHST}). For example,
the {\em resurgence}, defined as
$$\rho(I)=\sup\Big\{\frac{m}{r}: I^{(m)}\not\subseteq I^r\Big\},$$
 satisfies  (by  \cite[Theorem 1.2.1]{BH})
\begin{equation}
   \label{eq:estimate rho}
\frac{\alpha(I)}{\widehat{\alpha}(I)}\leq \rho(I).
\end{equation}

First we consider the change of Waldschmidt constants under specializations of matroid ideals. 

\begin{cor}
    \label{cor:Waldschmidt} 
Adopt the notation and assumptions of Theorem \ref{thm:matroid spec} and additionally assume that all forms $f_1,\ldots,f_s$ have the same degree, say $d$. Then we have the relation 
\[
\widehat{\alpha} (\ffi_*(I_{\Delta})) = d  \cdot \widehat{\alpha} (I_{\Delta}). 
\] 
\end{cor}

\begin{proof}
It is enough to observe that, for each monomial $\pi \in k[y_1,\ldots,y_s]$, $\deg \ffi (\pi) = d \cdot \deg (\pi)$.  
\end{proof} 

Again, we illustrate this result using hypersurface configurations. 

\begin{example}
    \label{exa:Waldschmidt uniform spec} The Stanley-Reisner ideal $I_{s,c}$ of a uniform matroid of dimension $s-c-1$ on $s$ vertices has Waldschmidt constant $\widehat{\alpha}(I_{s,c})=\frac{s}{c}$ by 
\cite{MFO} or \cite[Lemma 2.4.1, Lemma 2.4.2 and the proof of Theorem 2.4.3]{BH}. Specializing it by forms $f_1,\ldots,f_s$ of degree $d$, we get the ideal of a hypersurface configuration $V_{\lambda,c}$, where $\lambda = [d,\ldots,d]$. It follows that $\widehat{\alpha}(I_{V_{\lambda,c}})=\frac{d s}{c}$. 
\end{example} 

If we specialize by using forms of varying degree, things are more complicated. 
To compute $\alpha(\ffi_*(I_{\Delta})^{(m)})$ (and hence $\widehat{\alpha} (\ffi_*(I_{\Delta}))$), we must take all monomials in the $y_i$ which vanish
on all components of the variety defined by $I_{\Delta}$ to order at least $m$, and then find the minimum degree 
among these monomials after substituting $f_i$
in for each $y_i$. This is of course doable but will depend on the specific degrees of the $f_i$.

\begin{example} 
     \label{exa:non-uniform spec}
Consider specializations of  four coordinate points in $\PP^3$, that is, of the ideal $I_{4,3}$. The $m$-th symbolic power of its specialization is 
\[
\ffi_* (I_{4,3})^{(m)} = (f_1,f_2,f_3)^m \cap (f_1,f_2,f_4)^m \cap (f_1,f_3,f_4)^m \cap (f_2,f_3,f_4)^m. 
\]
Assume $m = 6 k$. As shown in Example \ref{exa:Waldschmidt uniform spec}, if all $f_i$ have degree $d$, then  $\widehat{\alpha} (\ffi_* (I_{4,3})) = \frac{4 d}{3}$. 
But suppose that $f_1,f_2,f_3$ are linear forms and $f_4$ has degree $d \ge 2$. Then $(f_1 f_2 f_3)^{3k}$ is in $\ffi_* (I_{4,3})^{(6 k)}$. In fact, $\ffi_* (I_{4,3})^{(3 k)}$ has initial degree $9 k$ in this case. Thus, the  Waldschmidt constant of $\ffi_* (I_{4,3})$  is 
\[
\widehat{\alpha} (\ffi_* (I_{4,3})) = \frac{9k}{6k} = \frac{3}{2}, 
\]
which is in fact $\widehat{\alpha}(I_{3,2})$ for the ideal $I_{3,2}\subseteq k[y_1,y_2,y_3]$. 
In particular, it is independent of the degree of $f_4$ whenever this degree is at least 2. 
\end{example}

We now turn attention to the resurgence. Proposition \ref{prop:inclusion-preserving} gives: 

\begin{cor}
    \label{cor:resurgence}
Adopting the notation and assumptions of Theorem \ref{thm:matroid spec}, we have: 
\begin{enumerate}

\item $\rho (\ffi_*( I_{\Delta})) \le \rho (I_{\Delta})$. 

\item If $f_1,\ldots,f_s$ is an $R$-regular sequence, then 
$\rho (\ffi_*( I_{\Delta})) = \rho (I_{\Delta})$.

\end{enumerate}
\end{cor} 

The second part of this result raises the following question: 

\begin{question}
    \label{q:res}
Does  the resurgence remain invariant  for any specialization of a matroid ideal as considered in Theorem \ref{thm:matroid spec}?
\end{question} 

Now we determine the resurgence in many new cases, giving further affirmative evidence for Question \ref{q:res}. 

\begin{thm} 
    \label{thm:resurgence spec uniform matroid} 
    
Assume that a sequence of homogeneous polynomials   $f_1,\ldots,f_s \in R$ satisfies one
of the following conditions: 
\begin{enumerate}

\item $f_1,\ldots,f_s \in R$ is an $R$-regular sequence. 

\item Any subset of at most $c+1$ of the forms $f_i$ forms an $R$-regular sequence, and all the forms $f_i$ have the same degree. 
\end{enumerate}
Consider the codimension $c$ hypersurface configuration $V_{\lambda, c} \subset \PP^n$ determined by $f_1,\ldots,f_s \in R$. Its ideal has resurgence
\[
\rho (I_{V_{\lambda, c}}) = \frac{c \cdot (s-c+1)}{s}.
\] 
\end{thm}    

This theorem is one of the few results which determines the resurgence of the ideal of a subscheme whose dimension is at least one and whose codimension is at least two, apart from  ideals of cones \cite[Proposition 2.5.1]{BH} and certain monomial ideals (see \cite[Theorem 1.5]{GHvT} and \cite[Theorem C]{LM}). 
In particular, the special case of Theorem \ref{thm:resurgence spec uniform matroid}, where all the forms $f_i$ are linear, gives the resurgence of every star configuration and thus answers  \cite[Question 4.12]{GHM3} affirmatively and extends \cite[Theorem C]{LM}.

\begin{proof}[Proof of Theorem \ref{thm:resurgence spec uniform matroid}]
Assume Condition (1) is satisfied, that is,  $I_{V_{\lambda, c}}$ is obtained by specializing the matroid ideal $I_{s,c}$, using the regular sequence $f_1,\ldots,f_s$. Then the result is a consequence of Corollary \ref{cor:resurgence}  and $\rho(I_{s,c}) =  \frac{c \cdot (s-c+1)}{s}$   \cite[Theorem C]{LM}. 

If Condition (2) is satisfied we argue similarly. Indeed, using also Corollary \ref{cor:Waldschmidt}, we get
\[
\frac{c \cdot (s-c+1)}{s} =  \frac{\alpha(I_{s,c})}{\widehat{\alpha}(I_{s,c})}  =  \frac{\alpha(I_{V_{\lambda, c}})}{\widehat{\alpha}(I_{V_{\lambda, c}})} \leq \rho(I_{V_{\lambda, c}}) \le \rho(I_{s, c}) = \frac{c \cdot (s-c+1)}{s}, 
\]
which yields our claim. 
\end{proof}

We now illustrate our results by considering specializations of  coordinate points. 

\begin{example}
     \label{exa:resurgence} 
If $f_0,\ldots,f_n$ is an $R$-regular sequence, then the ideal 
\[
\ffi_* (I_{n+1,n}) = \bigcap_{i=0}^n (f_0,\ldots,\hat{f_i},\ldots,f_n), 
\]
where $\hat{}$ indicates omitting, satisfies according to Theorem \ref{thm:resurgence spec uniform matroid}
\[
\rho (\ffi_* (I_{n+1,n})) = \frac{2 n}{n+1}. 
\]

Recall that in the case, where all the $f_i$ have the same degree, we have seen in the proof of Theorem \ref{thm:resurgence spec uniform matroid} that 
\[
\rho(\ffi_*(I_{n+1,n}))=\frac{\alpha(\ffi_*(I_{n+1,n}))}{\widehat{\alpha}(\ffi_*(I_{n+1,n}))}. 
\]
 Hence, Estimate \eqref{eq:estimate rho} is sharp in this case. However, if we  consider the situation in Example \ref{exa:non-uniform spec}, that is, $n = 3$ and $d_1 = d_2 =d_3 = 1$ and $d_4 = d \le 2$, then we get 
\[
\frac{\alpha(\ffi_*(I_{4,3}))}{\widehat{\alpha}(\ffi_*(I_{4,3}))} = \frac{2}{\frac{3}{2} }= \frac{4}{3} < \frac{3}{2} = \rho (\ffi_*(I_{4,3})). 
\]
\end{example} 
\medskip

As a final remark we want to draw attention to a remarkable connection between the configurations considered in this paper and a classical question in projective geometry.

To understand this connection   let $\lambda = [d_1, \ldots , d_s]$ be a partition of $d$.  The variety $\X_{n,\lambda} \subseteq \PP([R]_d) = \PP^{N-1}$, $N = 
\binom{d+n}{n}$, of $\lambda$-reducible forms of degree $d$ is defined by:
$$
\X_{n,d} := \{ [g] \in \PP^{N-1}\ | \ g = g_1 \cdots g_s,\ \deg g_i = d_i \} .
$$
These varieties have an interesting history and are discussed in detail in   \cite{Mamma},  \cite{CGGS} and \cite{CGGHMNS}.

One is interested in calculating the dimension of the (higher) secant varieties of this variety.  The famous Terracini Lemma explains that to do this one has to calculate the span of tangent spaces at general points of the variety.  So, it is important to know the tangent space at a general point of this variety. The remarkable fact is that if  $P =[f_1\cdots f_s]$ is a general point of $\X_{n,\lambda}$       then the projectivized tangent space at $P$      is the projectivization of the degree $d$ component of the ideal  $I$  which defines the codimension 2  $\lambda$-configuration associated to the forms $f_1, \ldots,f_s$  \cite[Proposition 3.2]{CChG}.


\end{document}